\documentclass[11pt]{article} 

\usepackage{amsmath}
\usepackage{lmodern}
\usepackage[T1]{fontenc}
\usepackage{babel}
\usepackage[babel=true]{microtype}
\usepackage{amsxtra}%
\usepackage{amsfonts}%
\usepackage{amssymb}%
\usepackage{amsthm}
\usepackage{graphicx}

\usepackage{color,hyperref}
\hypersetup{colorlinks,breaklinks,
	linkcolor=blue,urlcolor=blue,
	anchorcolor=blue,citecolor=blue}

\usepackage[margin=1in]{geometry}

\newcommand{\eps}{\varepsilon}

\newcommand{\R}{\mathbb{R}}

\newcommand{\C}{\mathbb{C}}

\renewcommand{\phi}{\varphi}

\newcommand{\mcl}{\mathcal{L}}
\newcommand{\spn}{\text{\span}}

\newcommand{\Rg}{\mathrm{Rg}}
\usepackage{enumerate}
\renewcommand{\Re}{\mathrm{Re} \,}

\newtheorem*{thm*}{Theorem}

\newtheorem{prop}{Proposition}
\newtheorem{lemma}[prop]{Lemma}
\newtheorem{corollary}[prop]{Corollary}

\newtheorem{thmlocal}[prop]{Theorem}

\newtheorem{hyp}{Hypothesis}

\newtheorem{remark}[prop]{Remark}

\numberwithin{equation}{section}
\numberwithin{prop}{section}

\renewcommand{\spn}{\mathrm{span}}

\renewcommand{\u}{\mathbf{u}}

\newcommand{\q}{\mathbf{q}}
\newcommand{\g}{\mathbf{g}}

\pdfoutput=1

\begin{document}
	\begin{center}
		{\fontsize{15}{15}\fontseries{b}\selectfont{Front propagation close to the onset of instability}}\\[0.2in]
		Montie Avery \\[0.1in]
		\textit{\footnotesize 
			Boston University, Department of Mathematics and Statistics, 665 Commonwealth Ave, Boston, MA, 02215\\
		}
	\end{center}
	
	\begin{abstract} 
		We describe the resulting spatiotemporal dynamics when a homogeneous equilibrium loses stability in a spatially extended system. More precisely, we consider reaction-diffusion systems, assuming only that the reaction kinetics undergo a transcritical, saddle-node, or supercritical pitchfork bifurcation as a parameter passes through zero. We construct traveling front solutions which describe the invasion of the now-unstable state by a nearby stable state. We show that these fronts are marginally spectrally stable near the bifurcation point, which, together with recent advances in the theory of front propagation into unstable states, establishes that these fronts govern the dynamics of localized perturbation to the unstable state. Our proofs are based on functional analytic tools to study the existence and eigenvalue problems for fronts, which become singularly perturbed after a natural rescaling. 
	\end{abstract}
	
	\section{Introduction}
	Invasion fronts play an important role in mediating transitions from one state to another in many physical systems. A common scenario is when a trivial background state becomes unstable through a bifurcation. Localized perturbations to this unstable state then grow and spread, and a new stable state is selected in the wake of this invasion process. A fundamental question is then to predict both the spreading speed and what new state is selected in the wake. 
	
	In the mathematics literature, front propagation into unstable states is often studied in scalar equations which admit comparison principles \cite{Bramson1, Bramson2, Uchiyama, Comparison1, Comparison2, Lau, HamelPeriodic, Graham, AnHendersonRyzhik1, AnHendersonRyzhik2}. However, many experiments on front propagation focus on more complex systems which do not admit comparison principles \cite{vanSaarloosReview}. The \emph{marginal stability conjecture} (see e.g. \cite{vanSaarloosMarginalStability, vanSaarloosReview} and references therein, or \cite{CAMS} for a recent mathematical perspective) provides a framework for predicting front invasion speeds in broad classes of systems, in particular without comparison principles. It asserts that selected invasion speeds are those for which there exists a corresponding traveling front solution which is marginally spectrally stable. The marginal stability conjecture was recently proved for systems of parabolic equations (including higher order parabolic equations and multi-component systems) in \cite{CAMS, AverySystems}. 
	
	The marginally stable spectrum associated to a selected front may be either essential spectrum or point spectrum. Since essential spectrum is invariant under compact perturbations, the dynamics in the former case are governed by the tail dynamics in the leading edge of the front, and so the fronts are said to be \emph{pulled}, and the associated speed is said to be linearly determined. In the former case, the dynamics are driven by the localized eigenmode near the front interface, and so the fronts are said to be \emph{pushed}, and the speed is said to be nonlinearly determined. 
	
	It is commonly suggested in the physics literature \cite{vanSaarloosReview} that when a trivial (that is, spatially constant) state loses stability through a supercritical bifurcation --- that is, a bifurcation in which another nearby state simultaneously becomes stable --- one should observe propagation of small amplitude, pulled fronts. The main result of the present work is to confirm this picture for the three most common types of bifurcation. 
	
	To that end, we consider reaction-diffusion systems,
	\begin{align}
		\u_t = D \u_{xx} + f(\u;\mu), \quad \u = \u(x,t) \in \R^n, \quad x \in \R, \quad t > 0, \label{e: RD}
	\end{align}
	with strictly positive definite diffusion matrix $D$ and smooth, parameter-dependent nonlinearity $f : \R^n \times \R \to \R^n$. We may refer to $f$ as the \emph{reaction kinetics}. Our main result can be summarized as follows. 
	\begin{thmlocal}\label{t: main}
		Assume that $f(\u; \mu)$ undergoes a transcritical, saddle-node, or supercritical pitchfork bifurcation at $(\u, \mu) = (0, 0)$, with all eigenvalues of the linearization $f_u (0; 0)$ negative except for the simple neutral eigenvalue associated with the bifurcation. Then \eqref{e: RD} admits marginally stable pulled front solutions, which describe the invasion of the now-unstable state. 
	\end{thmlocal}
	The present work is in part motivated by \cite{RaugelKirchgassner}, which considers dynamics of \eqref{e: RD} near a transcritical bifurcation, and establishes existence and some stability properties of fronts, though not selection from localized initial data. 
	
	The remainder of this paper is organized as follows. In Section \ref{s: marginal stability}, we explain how marginal stability leads to front selection from localized initial data by the results of \cite{CAMS, AverySystems}. In Sections \ref{s: transcrit}, \ref{s: saddle node}, and \ref{s: pitchfork}, we precisely formulate and prove Theorem \ref{t: main} for transcritical, saddle-node, and supercritical pitchfork bifurcations, respectively. The proof is quite similar in all three cases, so we give the full details in the transcritical case and explain the necessary modifications for the other cases. 
	
	\section{Front selection through marginal stability}\label{s: marginal stability} 
	Typically traveling front solutions $\u(x,t) = \q(x-ct;c)$ connecting to an unstable state at $x = + \infty$ exist for an open range of speeds $c$. Many of these fronts may be stable against perturbations which do not alter the tail decay of the traveling wave profile. However, according to the marginal stability conjecture, only those which are marginally spectrally stable attract initial conditions which vanish identically for $x$ sufficiently large. These initial conditions are naturally considered as a model for the propagation of compactly supported disturbances to the unstable state in \eqref{e: RD}, and so are the most relevant for applications. We refer to initial data which vanish for $x$ sufficiently large, or are otherwise very rapidly decaying, as \emph{steep}. 
	
	We now formulate precise assumptions which capture marginal spectral stability, and guarantee this attraction of steep initial data by the results of \cite{CAMS, AverySystems}. In this section, we will consider the parameter $\mu$ to be fixed, and write $f(\u;\mu) = f(\u)$. The speed of a pulled front may be predicted from the linearization about the unstable state in the leading edge, which we will take here to be $u \equiv 0$, so we consider the linearization
	\begin{align}
		\u_t = D \u_{xx} + c \u_x + f'(0) \u, \label{e: lin about 0}
	\end{align}
	in a moving frame with speed $c$. Dynamics of \eqref{e: lin about 0} may be analyzed through the \emph{dispersion relation}, 
	\begin{align}
		d_c (\lambda, \nu) = \det (D \nu^2 + c \nu I + f'(0) - \lambda I), \label{e: dispersion}
	\end{align}
	obtained via the Fourier-Laplace ansatz $\u(x,t) \sim e^{\nu x + \lambda t}$. In particular, if $\u$ is a solution to \eqref{e: lin about 0} with compactly supported initial data $\u_0$, then for each fixed $L > 0$ and $\eps > 0$, one has the pointwise growth bound \cite{HolzerScheelPointwiseGrowth}
	\begin{align}
		\sup_{x \in [-L, L]} |\u(x,t)| \leq C e^{(\Re \lambda_* + \eps )t}, \label{e: pointwise growth bound}
	\end{align}
	where $(\lambda_*, \nu_*)$ is the \emph{pinched double root} of the dispersion relation \eqref{e: dispersion} for which $\lambda_*$ has the largest real part. Pinched double roots are double roots in $\nu$, and so satisfy $d_c(\lambda_*, \nu_*) = \partial_\nu d_c (\lambda_*, \nu_*) = 0$. The term ``pinched'' refers to the fact that the continuations $\nu_\pm(\lambda)$ of the double root satisfy $\Re \nu_+(\lambda) \to \infty$ as $\Re \lambda \to \infty$, while $\Re \nu_-(\lambda) \to -\infty$ as $\Re \lambda \to \infty$. The pointwise growth bound can be proved using the inverse Laplace transform, and we refer to \cite{HolzerScheelPointwiseGrowth} for details. There are also earlier perspectives relying on the Fourier transform; see for instance \cite{vanSaarloosReview}. 
	
	In particular, \eqref{e: pointwise growth bound} suggests that marginal pointwise stability in the leading edge of the front is captured by marginal stability of pinched double roots, which we capture in the following hypothesis.
	\begin{hyp}[Linear spreading speed via pinched double root]\label{hyp: PDR}
		Assume that there exists $c_* > 0$ and $\nu_* < 0$ such that the dispersion relation $d_{c_*}(\lambda, \nu)$ with $c = c_*$ satisfies the following.  
		\begin{enumerate}[i)]
			\item (Simple pinched double root at the origin) For $\lambda, \tilde{\nu}$ near 0, we have the expansion
			\begin{align}
				d_{c_*}(\lambda, \nu_*+\tilde{\nu}) = d_{10} \lambda + d_{02} \tilde{\nu}^2 + \mathrm{O}(\lambda \tilde{\nu}, \tilde{\nu}^3, \lambda^2),
			\end{align}
			where $d_{10}, d_{02} \in \R$ satisfy $d_{10} d_{02} < 0$. 
			\item (Minimal marginal spectrum) If $d_{c_*} (i \omega, \nu_*+ ik) = 0$ for some $\omega, k \in \R$, then $\omega = k = 0$. 
			\item (No unstable spectrum) There are no solutions to $d_{c_*}(\lambda, \nu_*+ik) = 0$ with $k \in \R$ and $\Re \lambda > 0$. 
		\end{enumerate}
	\end{hyp}
	Condition i) of Hypothesis \ref{hyp: PDR} guarantees that the dispersion relation has a pinched double root at the origin. Together, conditions ii) and iii) imply that there are no other marginally stable or unstable pinched double roots. (Actually, they imply the slightly stronger condition that the essential spectrum of $\u \equiv 0$ is marginally stable in an exponentially weighted function space; see \cite{FayeHolzerScheelResonant} for an exploration of related subtleties.) We refer to $c_*$ as \emph{the linear spreading speed}. See \cite{HolzerScheelPointwiseGrowth} for further background on linear spreading speeds and pinched double roots. 
	
	To establish $c_*$ as the selected speed in the invasion process, we want to have a traveling front with speed $c_*$ which is marginally spectrally stable, which is captured in the next hypotheses. 
	\begin{hyp}[Existence of a critical front]\label{hyp: front existence}
		Assume that \eqref{e: RD} admits a traveling wave solution $\u(x,t) = \q_*(x-c_*t)$ satisfying
		\begin{align*}
			\lim_{\xi \to -\infty} \q_*(\xi) = \u_-, \quad \lim_{\xi \to \infty} \q_*(\xi) = 0
		\end{align*}
		for some selected state $\u_- \in \R^n$. Moreover, we assume that convergence to $\u_-$ is exponential as $\xi \to -\infty$, and as $\xi \to \infty$ we have the generic asymptotics
		\begin{align}
			\q_*(\xi) = [b (\u^0 \xi + \u^1) + a \u^0] e^{\nu_* \xi} + \mathrm{O}\left( e^{(\nu_*-\eta) \xi}\right), \label{e: front asymptotics}
		\end{align}
		for some $a,b \in \R$, $b \neq 0$ $\u^0, \u^1 \in \R^n$, and some $\eta > 0$. 
	\end{hyp}
	\begin{remark}
		Hypothesis \ref{hyp: PDR} implies that the linearization of the traveling wave formulation of \eqref{e: RD} with $c = c_*$ has a 2-by-2 Jordan block at the origin \cite{HolzerScheelPointwiseGrowth}. Hence, the asymptotics \eqref{e: front asymptotics} are generic under Hypothesis \ref{hyp: PDR}.  
	\end{remark}
	We want this critical front $\q_*$ to be marginally spectrally stable. The essential spectrum associated to dynamics near $\xi = + \infty$ will be marginally stable in an appropriate weighted space by Hypothesis \ref{hyp: PDR}. The left dispersion relation
	\begin{align*}
		d^-(\lambda, \nu) = \det (D \nu^2 + c_* \nu I + f'(\u_-) - \lambda I)
	\end{align*}
	determines the spectrum $\Sigma^-$ of the linearization about $\u_-$, in the moving frame with speed $c_*$, with
	\begin{align*}
		\Sigma^- = \{ \lambda \in \C : d^-(\lambda, ik) = 0 \text{ for some } k \in \R \}. 
	\end{align*}
	\begin{hyp}[Stability on the left]\label{hyp: left stability}
		Assume that $\Re (\Sigma^-) < 0$. 
	\end{hyp}
	Finally, we need to exclude unstable point spectrum. First, we define the exponential weight which (marginally) stabilizes the essential spectrum in the leading edge. Let $\eta_* = - \nu_*> 0$, and let $\omega_*: \R \to \R$ be a smooth positive weight function satisfying
	\begin{align*}
		\omega_*(\xi) = \begin{cases}
			1, & \xi \leq -1, \\
			e^{\eta_* \xi}, & \xi \geq 1. 
		\end{cases}
	\end{align*}
	Let $\mathcal{A}$ denote the linearization of \eqref{e: RD}, in the moving frame with speed $c_*$, about the front $\q_*$:
	\begin{align*}
		\mathcal{A} = D \partial_\xi^2 + c_* \partial_\xi + f'(\q_*). 
	\end{align*}
	Then, define the weighted linearization $\mcl$ through
	\begin{align*}
		\mcl u = \omega_* \mathcal{A} \left( \frac{\u}{\omega_*} \right). 
	\end{align*}
	The spectrum of $\mcl$ on $L^2(\R)$ is the same as the spectrum of $\mathcal{A}$ on the weighted space with norm $\| g \| = \| \omega_* \g \|_{L^2}$. It follows from Palmer's theorem \cite{Palmer1, Palmer2} and Hypotheses \ref{hyp: PDR} and \ref{hyp: left stability} that the essential spectrum of $\mcl$ is marginally stable. We exclude unstable point spectrum of $\mcl$ in the following hypothesis. 
	\begin{hyp}[No unstable point spectrum]\label{hyp: no unstable point spectrum}
		Assume that $\mcl : H^2 (\R) \subset L^2(\R) \to L^2(\R)$ has eigenvalues with $\Re \lambda \geq 0$. Moreover, we assume that there is no bounded solution to the equation $\mcl \u = 0$. 
	\end{hyp}
	If there were a bounded solution to $\mcl \u = 0$, this would signify that we are not fully in the pulled propagation regime, but at the transition point between pushed and pulled propagation; see \cite{AveryHolzerScheel} for further details. 
	
	To characterize propagation near the bifurcations studied here, we rely on the main result of \cite{AverySystems}, which establishes front selection from steep initial data under Hypotheses \ref{hyp: PDR} through \ref{hyp: no unstable point spectrum} as follows. First, given a weight parameter $r \in \R$, define a smooth positive algebraic weight $\rho_{0, r}$ satisfying
	\begin{align*}
		\rho_{0,r}(x) = \begin{cases}
			1, & x \leq -1, \\
			x^r, &x \geq 1. 
		\end{cases}
	\end{align*}
	\begin{thmlocal}[{\cite{AverySystems}}]\label{t: selection}
		Assume Hypotheses \ref{hyp: PDR} through \ref{hyp: no unstable point spectrum} hold. Fix $2 < r < \frac{17}{8}$. The critical front $\q_*$ is \emph{selected} in the sense of \cite[Definition 1]{CAMS}. More precisely, for each $\eps > 0$, there exists a set of initial data $\mathcal{U}_\eps \subset L^\infty (\R)$ such that the following hold. 
		\begin{enumerate}
			\item For each $\u_0 \in \mathcal{U}_\eps$, the solution $\u$ to \eqref{e: RD} with initial data $\u_0$ satisfies 
			\begin{align}
				\sup_{x \in \R} | \rho_{0,-1} (x) \omega_*(x) [\u(x+\sigma(t), t) - \q_*(x)]| < \eps,
			\end{align}
			for all $t \geq t_* (\u_0)$, sufficiently large, where 
			\begin{align}
				\sigma(t) = c_* t - \frac{3}{2 \eta_*} \log t + x_\infty (\u_0)
			\end{align}
			for some $x_\infty(\u_0) \in \R$. 
			\item $\mathcal{U}_\eps$ contains some steep initial data. More precisely, there exists $\u_0 \in \mathcal{U}_\eps$ such that $\u_0(x) \equiv 0$ for $x$ sufficiently large. 
			\item $\mathcal{U}_\eps$ is open in the topology induced by the norm $\| \g \| = \| \rho_{0,r} \omega_* \g \|_{L^\infty}$.  
		\end{enumerate}
	\end{thmlocal}
	Thus, to prove Theorem \ref{t: main}, our goal is to show that Hypotheses \ref{hyp: PDR} through \ref{hyp: no unstable point spectrum} are satisfied when the reaction kinetics $f(\u; \mu)$ undergo a transcritical, saddle-node, or supercritical pitchfork bifurcation. 
	
	A key tool will be Fredholm properties of $\mcl$ on exponentially weighted spaces, implied by Hypotheses \ref{hyp: PDR} and \ref{hyp: left stability}. Given rates $\eta_{\pm} \in \R$, we define a smooth, positive two-sided exponential weight $\omega_{\eta_-, \eta_+}$ satisfying
	\begin{align}
		\omega_{\eta_-, \eta_+} (\xi) = \begin{cases}
			e^{\eta_- \xi}, & \xi \leq -1, \\
			e^{\eta_+ \xi}, & \xi \geq 1. 
		\end{cases}
	\end{align}
	Given a non-negative integer $k$, we define the exponentially weighted Sobolev space $H^k_{\eta_-, \eta_+} (\R, \C^n)$ through the norm 
	\begin{align}
		\| \g \|_{H^k_{\eta_-, \eta_+}} = \| \omega_{\eta_-, \eta_+} \g \|_{H^k}. 
	\end{align}
	When $k = 0$, we write $H^0_{\eta_-,\eta_+} = L^2_{\eta_-,\eta_+}$. 
	
	We will repeatedly use the following Fredholm properties of $\mcl$, which follow from Hypotheses \ref{hyp: PDR} and \ref{hyp: left stability} together with Palmer's theorem relating Fredholm properties to exponential dichotomies and asymptotic Morse indices \cite{Palmer1, Palmer2}. 
	\begin{lemma}\label{l: fredholm properties}
		Assume Hypotheses \ref{hyp: PDR} through \ref{hyp: left stability} hold. Fix $\eta> 0$ sufficiently small, and consider $\mcl$ as an operator $\mcl : H^2_{0, \eta} \to L^2_{0, \eta}$. Then $\mcl$ is a Fredholm operator with index -1. 
	\end{lemma}
	
	\section{Transcritical bifurcation}\label{s: transcrit}
	Consider \eqref{e: RD} near a transcritical bifurcation, with normal form 
	\begin{align}
		u_t &= u_{xx} + \mu u - u^2 + f_0 (u, v; \mu) \nonumber\\
		v_t &= D v_{xx} - K v + f_1 (u, v; \mu), \label{e: transcritical 2}
	\end{align}
	for $(u, v) \in \R \times \R^{n-1}$ with small parameter $\mu > 0$. We assume that $D, K \in \R^{n-1 \times n-1}$ each have strictly positive eigenvalues. We assume that $f_0$ and $f_1$ are smooth, with
	\begin{align}
		f_0 (u, v; \mu) &= \mathrm{O}\left(\mu^2 u^2, u |v|, |v|^2, u^3 \right), \nonumber \\
		f_1 (u, v; \mu) &= \mathrm{O}\left(\mu |v|, u^2, |v|^2, u|v| \right) \label{e: transcritical nonlinearity expansions 2}
	\end{align}
	as $\mu, u, |v| \to 0$. We further assume that 
	\begin{align}
		\det (- D k^2 - K- \lambda I) \neq 0 \text{ for any } k \in \R, \lambda \in \C \text{ with } \Re \lambda \geq 0.  \label{e: no turing}
	\end{align}
	This assumption ensures that the $v$ component does not undergo any Turing-type bifurcation, which would introduce a secondary instability. Note that in particular that this, together with separate invertibility of $D$ and $K$, implies that there is a constant $C > 0$ such that
	\begin{align}
		\sup_{k \in \R} |(-Dk^2 - K)^{-1} | \leq C, \label{e: transcritical uniform bound symbol}
	\end{align}
	where $|\cdot|$ is some fixed matrix norm. 
	
	In \cite{RaugelKirchgassner}, it was shown via a center manifold reduction that systems of this type (with $D = I$) admit critical pulled front solutions, and using energy estimates the authors showed that these pulled fronts are nonlinearly stable in certain weighted spaces. Here, we further show that these fronts attract open classes of steep initial data by showing that Hypotheses \ref{hyp: PDR}-\ref{hyp: no unstable point spectrum} are automatically satisfied near the transcritical bifurcation. Since existence of critical fronts, with weak exponential decay \eqref{e: front asymptotics} was already shown in \cite{RaugelKirchgassner}, the main contribution here is to verify marginal spectral stability of these fronts. This was not needed for the nonlinear stability argument in \cite{RaugelKirchgassner}, which relied on energy estimates, but here will imply selection of pulled fronts from steep initial data by Theorem \ref{t: selection}. 
	
	We give a unified approach to existence and spectral stability of these fronts following that of \cite{AveryGarenaux}, which established existence and marginal spectral stability of pulled fronts in the extended Fisher-KPP equation. As in \cite{RaugelKirchgassner}, we first introduce the rescaled variables
	\begin{align}
		y = \sqrt{\mu} x, \quad \tau = \mu t, \quad U(y, \tau) = \mu^{-1} u(x,t), \quad V(y, \tau) = \mu^{-1} v(x,t). 
	\end{align}
	The new unknowns $U$ and $V$ then solve the system
	\begin{align}
		U_\tau &= U_{yy} + U - U^2 + g_0 (U, V; \mu), \nonumber\\
		\mu V_\tau &= \mu D V_{yy} - K V +  \mu g_1 (U, V; \mu), \label{e: transcritical scaled 2}
	\end{align}
	where
	\begin{align}
		g_0 (U, V; \mu) := \frac{1}{\mu^2} f_0 (\mu U, \mu V; \mu), \quad g_1 (\mu) = \frac{1}{\mu^2} f_1 (\mu U, \mu V; \mu)
	\end{align}
	are smooth in all arguments by \eqref{e: transcritical nonlinearity expansions 2}. 
	
	\begin{thmlocal}\label{t: transcritical}
		For $\mu > 0$ sufficiently small, the system \eqref{e: transcritical scaled 2} satisfies Hypotheses \ref{hyp: PDR} through \ref{hyp: no unstable point spectrum}. 
	\end{thmlocal}
	
	The remainder of this section is dedicated to proving Theorem \ref{t: transcritical}. We first compute the linear spreading speed.
	\begin{lemma} \label{l: transcrit hyp 1}
		For $\mu > 0$, the system \eqref{e: transcritical scaled 2} satisfies Hypothesis \ref{hyp: PDR}, with $c_* = 2, \eta_* = 1$. 
	\end{lemma}
	\begin{proof}
		Passing to a moving frame with speed $c$ and linearizing about $(U,V) = 0$, we find the dispersion relation
		\begin{align}
			d_c (\lambda, \nu; \mu) &= \det \begin{pmatrix}
				\nu^2 + c \nu + 1 - \lambda I & 0 \\
				0 & D \mu \nu^2 + c \mu \nu I -  K + f_1^{011} \mu - \lambda I
			\end{pmatrix} \\
			&= (\nu^2 + c \nu + 1) \det (D \mu \nu^2 + c \mu \nu I -  K + f_1^{011} \mu - \lambda I),
		\end{align}
		which has a simple double root at $(\lambda, \nu) = (0, -1)$ for $c = c_* := 2$. Here $f_1^{011} = \partial_v \partial_\mu f_1 (0,0;0)$. 
		
		We now verify that the essential spectrum is otherwise stable. First, note that the real part of the spectrum of $D \mu \partial_y^2 + c \mu \partial_y - K$ coincides with the real part of the spectrum of $D \mu \partial_y^2 - K$. We determine the essential spectrum of the latter operator by taking the Fourier transform, and introducing the scalings $\kappa = \sqrt{\mu} k$ and $\tilde{\lambda} = \mu \lambda$. We then see
		\begin{align}
			\det (- \mu D k^2 - K - \mu \lambda I) = \det (- D \kappa^2 - K - \tilde{\lambda} I). 
		\end{align}
		By \eqref{e: no turing}, the latter polynomial has no roots $\kappa \in \R, \tilde{\lambda} \in \C$ with $\Re \tilde{\lambda} \geq 0$, which implies the desired result since $\lambda = \mu^{-1} \tilde{\lambda}$ with $\mu > 0$. 
	\end{proof}
	
	We now determine the selected state in the wake of the invasion process.
	\begin{lemma}\label{l: transcritical stability on left}
		The system \eqref{e: transcritical scaled 2} admits a spatially uniform equilibrium solution 
		\begin{align}
			W_* (\mu) = \begin{pmatrix} U_*(\mu) \\ V_*(\mu) \end{pmatrix} = \begin{pmatrix} 1 \\ 0 \end{pmatrix} + \mathrm{O}(\mu),
		\end{align}
		which is smooth in $\mu$. Furthermore, the essential spectrum of the linearization of \eqref{e: transcritical scaled 2} about $W_*(\mu)$ is strictly stable. 
	\end{lemma}
	\begin{proof}
		Spatially constant equilibria to \eqref{e: transcritical scaled 2} for $\mu > 0$ solve
		\begin{align}
			0 &= U - U^2 + \frac{1}{\mu^2} f_0 (\mu U, \mu V; \mu) = U - U^2 + f_0^{110} U V + f_0^{020} V^2 + \mathrm{O}(\mu) \\
			0 &= -K V + \frac{1}{\mu} f_1 (\mu U, \mu V; \mu) = - KV + \mathrm{O}(\mu),
		\end{align}
		for some constants $f^{110}_0, f_0^{020} \in \R$, using the expansions \eqref{e: transcritical nonlinearity expansions 2}. We then find a solution $(U, V)^T = (1, 0)^T$ at $\mu = 0$. One readily verifies that the linearization at this solution is invertible, and so the existence of $W_*(\mu)$ follows from the implicit function theorem. Using \eqref{e: no turing}, one finds that the essential spectrum of the linearization about $W_*(\mu)$ is stable, as desired. 
	\end{proof}
	
	We now prove the existence of pulled fronts traveling with the linear spreading speed $c_* = 2$ near the transcritical bifurcation. Such fronts solve the traveling wave equation
	\begin{align}
		U_{yy} + 2 U_y + U - U^2 + g_0 (U, V; \mu) &= 0, \nonumber\\
		\mu D V_{yy} + 2 \mu V_y -  K V + \mu g_1 (U, V; \mu) &= 0. \label{e: transcrit TW2}
	\end{align}
	Intuitively, when $\mu$ is small the second equation should imply $V \approx 0$ by invertibility of $K$, so that existence and properties of fronts can then be recovered from the first equation. The perturbation in \eqref{e: transcrit TW2} to $\mu \neq 0$ is singular, however. We overcome this by using appropriately chosen preconditioners to regularize the singular perturbation, as in \cite{AveryGarenaux, GohScheel, JensArnd}. 
	
	To construct invasion fronts, we make the far-field/core ansatz
	\begin{align}
		\begin{pmatrix}
			U(x) \\ V(x) 
		\end{pmatrix} =
		\begin{pmatrix}
			U_\mathrm{ff/c} (x) \\  V_\mathrm{ff/c} (x)
		\end{pmatrix} :=
		W_*(\mu) \chi_- (x) + \begin{pmatrix} w_U \\ w_V \end{pmatrix} + e_0 \chi_+ (x) (a + x) e^{-x}, \label{e: transcrit ff core ansatz}
	\end{align}
	where $e_0 = (1, 0)^T \in \R\times \R^{n-1}$ and $a \in \R$. Inserting this ansatz into \eqref{e: transcrit TW2}, we find an equation
	\begin{align}
		F (w_U, w_V, a; \mu) = 0, 
	\end{align}
	where $F$ is defined by
	\begin{align}
		F (w_U, w_V, a; \mu) = \begin{pmatrix}
			(\partial_{yy} + 2 \partial_y + 1) U_\mathrm{ff/c} - U_\mathrm{ff/c}^2 + g_0 ( U_\mathrm{ff/c},  V_\mathrm{ff/c}; \mu) \\
			(\mu D \partial_{yy} - K)  V_\mathrm{ff/c} + 2 \mu \partial_y V_\mathrm{ff/c} + \mu g_1 ( U_\mathrm{ff/c},  V_\mathrm{ff/c}; \mu)
		\end{pmatrix}. 
	\end{align}
	The singularly perturbed structure of \eqref{e: transcritical scaled 2} presents an obstacle in choosing a consistent domain for $F$, since the second component of $F$ involves two spatial derivatives for $\mu \neq 0$, but no derivatives for $\mu = 0$. We overcome this by letting $\mu = \delta^2$ and defining the regularized function
	\begin{align}
		G(w_U, w_V, a; \delta) = \begin{pmatrix}
			I & 0 \\ 0 & (\delta^2 D \partial_{yy} - K)^{-1} 
		\end{pmatrix}
		F (w_U, w_V, a; \delta^2). \label{e: transcrit G def}
	\end{align}
	We fix $\delta_0 > 0$ small and $\eta = 1 + \tilde{\eta}$ with $\tilde{\eta}$ small, and consider $G$ as a function
	\begin{align}
		G : H^2_{\mathrm{exp}, 0, \eta} \times H^1_{\mathrm{exp}, 0, \eta} \times (-\delta_0, \delta_0) \to L^2_{\mathrm{exp}, 0, \eta} \times H^1_{\mathrm{exp}, 0, \eta}. 
	\end{align}
	
	To prove that $G$ is well-defined on these spaces, we will need the following estimates on the preconditioner $(\delta^2 D \partial_y^2 - K)^{-1}$. 
	\begin{lemma}
		Fix a non-negative integer $m$. There exist positive constants $\delta_0$, and $C = C(\delta_0, m)$ such that
		\begin{align}
			\| (\delta^2 D \partial_{yy} - K)^{-1} \|_{H^m \to H^m} &\leq C, \label{e: preconditioner 1} \\
			\| (\delta^2 D \partial_{yy} - K)^{-1} \|_{H^m \to H^{m+1}} &\leq \frac{C}{|\delta|}. \label{e: preconditioner 2} 
		\end{align}
		for all $|\delta| < \delta_0$.
	\end{lemma}
	\begin{proof}
		By Plancherel's theorem, we have
		\begin{align}
			\| (\delta^2 D \partial_y^2 - K)^{-1} f \|_{H^m} &= \left\| k \mapsto \langle k \rangle^m (-\delta^2 D k^2 - K)^{-1} \hat{f} (k) \right\|_{L^2} \\
			&\leq \sup_{k \in \R} | (-\delta^2 D k^2 - K)^{-1} | \| f \|_{H^m},
		\end{align}
		where $\hat{f}$ denotes the Fourier transform of $f$ and $|\cdot|$ is the induced matrix norm from the Euclidean norm. Introducing $\kappa = \delta k$, we find
		\begin{align}
			\sup_{k \in \R} | (-\delta^2 D k^2 - K)^{-1}| = \sup_{\kappa \in \R} | (- D \kappa^2 - K)^{-1}| \leq C
		\end{align}
		for all $\delta$ small, where the last estimate follows from \eqref{e: no turing} together with the fact that the resolvent operator of a bounded operator (in particular, a matrix) is uniformly bounded for large spectral parameter. This establishes \eqref{e: preconditioner 1}. 
		
		To prove \eqref{e: preconditioner 2}, we first estimate in $H^{m+2}$ and then interpolate. By Plancherel's theorem, we have
		\begin{align*}
			\left\| k \mapsto (-\delta^2 D k^2 - K)^{-1} \hat{f} (k) \right\|_{H^{m+2}} \leq \sup_{k \in \R} \left( \langle k \rangle^2 | (-\delta^2 D k^2 - K)^{-1}| \right) \| f \|_{H^m}. 
		\end{align*}
		Again introducing $\kappa = \delta k$, we find
		\begin{align*}
			\langle k \rangle^2 (-\delta^2 D k^2 - K)^{-1} = (-D \kappa^2 - K)^{-1} + \frac{\kappa^2}{\delta^2} (- D\kappa^2 - K)^{-1}. 
		\end{align*}
		The first term is bounded by \eqref{e: no turing}, while for the second term, using the Neumann series expansion at $ -\kappa^2 =\infty$ of the resolvent $(D^{-1} K + \kappa^2)^{-1}$, we have
		\begin{align}
			\frac{\kappa^2}{\delta^2} |(- D\kappa^2 - K)^{-1}| \leq \frac  {\kappa^2}{\delta^2} \frac{C}{\kappa^2} \leq \frac{C}{\delta^2},
		\end{align}
		from which we conclude
		\begin{align}
			\left\| k \mapsto (-\delta^2 D k^2 - K)^{-1} \hat{f} (k) \right\|_{H^{m+2}} \leq \frac{C}{\delta^2} \| f \|_{H^m}. \label{e: H2 estimate}
		\end{align}
		The estimate \eqref{e: preconditioner 2} then follows from interpolating \eqref{e: H2 estimate} and \eqref{e: preconditioner 1}. 
	\end{proof}
	
	We now extend the preconditioner estimates \eqref{e: preconditioner 1}-\eqref{e: preconditioner 2} to exponentially weighted spaces. 
	
	\begin{lemma}\label{l: preconditioner regularity}
		Fix a non-negative integer $m$. There exist positive constants $\delta_0, \eta_0$ and $C = C(\delta_0, \eta_0, m)$ such that 
		\begin{align}
			\| (\delta^2 D \partial_{yy} - K)^{-1} \|_{H^m_{\mathrm{exp}, 0, \eta} \to H^m_{\mathrm{exp}, 0, \eta}} &\leq C, \label{e: preconditioner weighted 1} \\
			\| (\delta^2 D \partial_{yy} - K)^{-1} \|_{H^m_{\mathrm{exp}, 0, \eta} \to H^{m+1}_{\mathrm{exp}, 0, \eta}} &\leq \frac{C}{|\delta|}. \label{e: preconditioner weighted 2} 
		\end{align}
		provided $|\delta| \leq \delta_0$ and $\eta = 1 + \tilde{\eta}$ with $|\tilde{\eta}| \leq \eta_0$. 
	\end{lemma}
	\begin{proof}
		Note that $H^m_{\mathrm{exp}, 0, \eta} = H^m \cap H^m_{\mathrm{exp}, \eta, \eta}$, with equivalence of norms
		\begin{align*}
			\| f \|_{H^m_{\mathrm{exp}, 0, \eta}} \sim \| f \|_{H^m} + \| f \|_{H^m_{\mathrm{exp}, \eta, \eta}},
		\end{align*}
		so it suffices to prove the estimates on $H^m_{\mathrm{exp}, \eta, \eta}$ for all $|\eta| \leq 2$. The advantage of considering these spaces instead is that multiplication by $e^{-\eta x}$ is an isomorphism from $H^m$ to $H^m_{\mathrm{exp}, \eta, \eta}$, so that it suffices to prove $L^2$-estimates on the conjugate operator $(\delta^2 D (\partial_y - \eta)^2 - K)^{-1}$. 
		
		Having already established the estimates for $\eta = 0$, we separate out this principle part, writing
		\begin{align*}
			\delta^2 D (\partial_y - \eta)^2 - K = [\delta^2 D \partial_y^2 - K] - [2 \delta^2 \eta D \partial_y + \delta^2 \eta^2 D] =: T_0(\delta) + \tilde{T}(\eta, \delta). 
		\end{align*}
		To take advantage of the already established invertibility of $T_0$, we write
		\begin{align*}
			T_0(\delta) + \tilde{T}(\eta, \delta) = [I + \tilde{T}(\eta, \delta) T_0(\delta)^{-1}] T_0(\delta)
		\end{align*}
		Assuming for now invertibility, the inverse of this operator is given by
		\begin{align}
			(T_0(\delta) + \tilde{T}(\eta, \delta))^{-1} = T_0(\delta)^{-1} (I + \tilde{T}(\eta, \delta) T_0(\delta)^{-1})^{-1}. \label{e: precond factoring}
		\end{align}
		Note that $\| \tilde{T} (\eta, \delta) \|_{H^{m+1} \to H^m} \leq C \delta^2$ for $|\eta| \leq 2$, and by \eqref{e: preconditioner 1} we have $\| T_0 (\delta)^{-1} \|_{H^m \to H^{m+1}} \leq C |\delta|^{-1}$ for $\delta$ small. Hence
		\begin{align*}
			\| \tilde{T}(\eta, \delta) T_0(\delta)^{-1}\|_{H^m \to H^m} \leq C |\delta|
		\end{align*}
		for $\delta$ and $|\eta| \leq 2$. Hence we can invert $(I + \tilde{T}(\eta, \delta) T_0(\delta)^{-1})$ in $H^m$ with the geometric series, and the inverse is uniformly bounded from $H^m$ to $H^m$ for $\delta$ small and $|\eta| \leq 2$. Using \eqref{e: precond factoring}, we then obtain
		\begin{align*}
			\| (T_0(\delta) + \tilde{T}(\eta, \delta))^{-1} \|_{H^m \to H^m} \leq \| T_0 (\delta)^{-1} \|_{H^m \to H^m} \| (I + \tilde{T}(\eta, \delta) T_0(\delta)^{-1})^{-1} \|_{H^m \to H^m} \leq C,
		\end{align*}
		and 
		\begin{align}
			\| (T_0(\delta) + \tilde{T}(\eta, \delta))^{-1} \|_{H^m \to H^{m+1}} \leq \| T_0 (\delta)^{-1} \|_{H^m \to H^{m+1}} \| (I + \tilde{T}(\eta, \delta) T_0(\delta)^{-1})^{-1} \|_{H^m \to H^m} \leq \frac{C}{|\delta|}
		\end{align}
		for $\delta$ small and $|\eta| \leq 2$. 
	\end{proof}
	
	\begin{corollary}
		There exist $\delta_0 > 0$ and $\tilde{\eta}> 0$ sufficiently small, such that for $\eta = 1 + \tilde{\eta}$, the mapping $G : H^2_{\mathrm{exp}, 0, \eta} \times H^1_{\mathrm{exp}, 0, \eta} \times \R \times (-\delta_0, \delta_0) \to L^2_{\mathrm{exp}, 0, \eta} \times H^1_{\mathrm{exp}, 0, \eta}$ is well-defined, smooth in $w_U, w_V,$ and $a$, and continuous in $\delta$. 
	\end{corollary}
	\begin{proof}
		That $G$ preserves exponential localization follows from the fact that the far-field term $e_0 \chi_+(x) (a+x)e^{-x}$ in \eqref{e: transcrit ff core ansatz} solves \eqref{e: transcrit TW2} up to a residual error of size $\mathrm{O}(x^2 e^{-2x})$ arising from the nonlinear terms. Smoothness in $w_U$ and $w_V$ follows from the fact that $H^m_{\mathrm{exp}, 0, \eta}$ is a Banach algebra for $m = 1, 2$. Continuity in $\delta$ follows from the estimates of Lemma \ref{l: preconditioner regularity}. 
	\end{proof}
	
	At $\delta = 0$, \eqref{e: transcrit TW2} has a solution $(U(y), V(y))^T = (q_0(y), 0)^T$, where $q_0$ is the critical Fisher-KPP front, solving
	\begin{align}
		q_0'' + 2 q_0' + q_0 - q_0^2 = 0, \quad \lim_{y \to -\infty} q_0 (y) = 1, \quad \lim_{y \to \infty} q_0 (y) = 0. 
	\end{align} 
	The front $q_0$ has asymptotics $q_0(y) \sim (a + by) e^{-y}, y \to \infty$, but by translating in space we can assume $b = 1$, changing the value of $a =: a_0$. We therefore find a corresponding solution $G(w_U^0, 0, a_0; 0)$ with
	\begin{align}
		w_U^0 (y) = q_0 (y) - \chi_- (y) - \chi_+(y) (a_0 + y) e^{-y}. 
	\end{align}
	Sicne Fisher-KPP fronts satisfy Hypotheses \ref{hyp: PDR} through \ref{hyp: no unstable point spectrum} (see e.g. \cite{AveryGarenaux}), it follows from Lemma \ref{l: fredholm properties} that the linearization $D_w G (w_U^0, 0, a_0; 0)$ in $w = (w_U, w_V)^T$ is Fredholm with index $-1$. By the Fredholm bordering lemma (see e.g. \cite[Lemma 4.4]{ArndBjornRelativeMorse}), the joint linearization $D_{(w, a)} G (w_U^0, 0, a_0; 0)$ is Fredholm with index 0. 
	\begin{lemma}\label{l: transcrit lin invertible}
		Fix $\tilde{\eta}$ small and let $\eta = 1 + \tilde{\eta}$. The joint linearization $D_{(w,a)} G (w_U^0, 0, a_0; 0) :  H^2_{\mathrm{exp}, 0, \eta} \times H^1_{\mathrm{exp}, 0, \eta} \times \R^2 \to L^2_{\mathrm{exp}, 0, \eta} \times H^1_{\mathrm{exp}, 0, \eta}$ is invertible.
	\end{lemma}
	\begin{proof}
		Since the linearization is Fredholm index 0, it suffices to show that the kernel is trivial. From a short calculation, we find that this linearization is given by
		\begin{align*}
			D_{(w,a)} G (w_U^0, 0, a_0; 0) = \begin{pmatrix}
				\mathcal{A}_\mathrm{kpp} & f_0^{110} q_0 & \mathcal{A}_\mathrm{kpp} [\chi_+ e^{-\cdot}]  \\
				0 & I & 0
			\end{pmatrix},
		\end{align*}
		where
		\begin{align}
			\mathcal{A}_\mathrm{kpp} = \partial_y^2 + 2 \partial_y + 1 - 2 q_0. 
		\end{align}
		is the linearization about the critical Fisher-KPP front. 
		Suppose that $(u_0, v_0, \alpha) \in \ker D_{(w,a)} (w_U^0, 0, a_0; 0)$. We immediately see that $v_0 = 0$, and so we must have
		\begin{align}
			\mathcal{A}_\mathrm{kpp} (u_0 + \alpha \chi_+ e^{-x}) = 0. 
		\end{align}
		with $u_0 \in H^1_{\mathrm{exp}, 0, \eta}$. 
		If $u_0$ or $\alpha$ were nonzero, then we would have a solution to $\mathcal{A}_\mathrm{kpp} u =0$ for which $\omega_{0,1} u$ is bounded, but there are no such solutions to this equation: one solution $q_0'$ comes from the translational mode and satisfies $q_0'(y) \sim y e^{-y}, y \to \infty$, and the other is exponentially growing at $-\infty$. Hence we conclude $u_0 = \alpha = 0$, and so the kernel is trivial, as desired. 
	\end{proof}
	
	\begin{corollary}\label{c: transcrit existence}
		There exists $\eta_0 > 0$ such that for $\mu = \delta^2$ with $\delta$ small, the system \eqref{e: transcrit TW2} admits front solutions $(U_\mathrm{fr}(y; \delta), V_\mathrm{fr}(y; \delta)^T$ satisfying
		\begin{align}
			\lim_{y \to -\infty} \begin{pmatrix} U_\mathrm{fr}(y; \delta) \\ V_\mathrm{fr}(y; \delta) \end{pmatrix} &= W_*(\delta^2), \\
			\begin{pmatrix} U_\mathrm{fr}(y; \delta) \\ V_\mathrm{fr}(y; \delta) \end{pmatrix} &= \begin{pmatrix} (a(\delta)+y)e^{-y} \\ 0 \end{pmatrix} + \mathrm{O}(e^{-(1+\eta_0) y}), \quad y \to \infty. 
		\end{align}
		where $a(\delta)$ is continuous in $\delta$ and satisfies $a(0) = a_0$. In particular, the system \eqref{e: transcritical scaled 2} satisfies Hypotheses \ref{hyp: front existence} and \ref{hyp: left stability}. 
	\end{corollary}
	\begin{proof}
		By Lemma \ref{l: transcrit lin invertible}, we can solve $G(w_U, w_V, a; \delta) = 0$ in a neighborhood of $(w_U^0, 0, a_0; 0)$ with the implicit function theorem, which establishes the existence and asymptotics of the fronts. Stability of the essential spectrum in the wake was already proven in Lemma \ref{l: transcritical stability on left}. 
	\end{proof}
	
	It only remains to verify Hypothesis \ref{hyp: no unstable point spectrum}. Together with our regularization of the singular perturbation, we use methods developed in \cite{PoganScheel} to construct a scalar function which detects eigenvalues near the essential spectrum. Let $\mathcal{A}(\delta)$ denote the linearization
	\begin{align}
		\mathcal{A}(\delta) &= \begin{pmatrix}
			\partial_y^2 + 2 \partial_y + 1 - 2 U_\mathrm{fr} + \partial_U g_0 (U_\mathrm{fr}, V_\mathrm{fr}; \delta^2) & \partial_V g_0 (U_\mathrm{fr}, V_\mathrm{fr}; \delta^2) \\  
			\delta^2 \partial_U g_1 (U_\mathrm{fr}, V_\mathrm{fr}; \delta^2) & \delta^2 D \partial_y^2 + 2 \delta^2 \partial_y - K + \delta^2 \partial_V  g_1 (U_\mathrm{fr}, V_\mathrm{fr}; \delta^2)
		\end{pmatrix} \\
		&= :\begin{pmatrix}
			\mathcal{A}_{11}(\delta) & \mathcal{A}_{12}(\delta) \\
			\mathcal{A}_{21}(\delta) & \mathcal{A}_{22} (\delta)
		\end{pmatrix}
	\end{align}
	about the front $(U_\mathrm{fr}, V_\mathrm{fr})^T$. As $y \to +\infty$, this limits on the far-field linearization
	\begin{align}
		\mathcal{A}_+(\delta) = \begin{pmatrix}
			\partial_y^2 + 2 \partial_y + 1 & 0 \\
			0 & \delta^2 D \partial_y^2 + 2 \delta^2 \partial_y - K 
		\end{pmatrix}. 
	\end{align}
	The limiting eigenvalue problem $(\mathcal{A}_+(\delta) - \gamma^2) (U,V)^T = 0$ then admits solutions
	\begin{align}
		\begin{pmatrix} U(y, \gamma) \\ V(y, \gamma) \end{pmatrix} = \begin{pmatrix} 
			e^{-(1\pm\gamma) y} \\ 0
		\end{pmatrix}
	\end{align}
	associated to the pinched double root at $(\lambda, \nu) = (0, -1)$. 
	
	The full eigenvalue problem has the form
	\begin{align}
		(\mathcal{A}_{11}(\delta) - \gamma^2) U + \mathcal{A}_{12}(\delta) V &= 0 \\
		\mathcal{A}_{21}(\delta) U + (\mathcal{A}_{22}(\delta) - \gamma^2) V&=0.
	\end{align}
	Fix $\eta = 1 + \tilde{\eta}$ with $\tilde{\eta}$ small. Modifying the argument of Lemma \ref{l: preconditioner regularity} to include $\gamma$-dependence, we see that $\delta$ and $\gamma$ sufficiently small, the operator $(\mathcal{A}_{22}(\delta) - \gamma^2)$ is invertible with inverse uniformly bounded from $L^2_{\mathrm{exp}, 0, \eta} \to L^2_{\mathrm{exp}, 0, \eta}$, and so we can solve the second equation for $V$ in terms of $U$, obtaining
	\begin{align}
		V = - (\mathcal{A}_{22}(\delta) - \gamma^2)^{-1} \mathcal{A}_{21}(\delta) U. 
	\end{align}
	Inserting this into the first equation, we obtain the nonlocal generalized eigenvalue problem
	\begin{align}
		(\mathcal{A}_{11}(\delta) - \gamma^2) U - \mathcal{A}_{12}(\delta) (\mathcal{A}_{22}(\delta) - \gamma^2)^{-1} \mathcal{A}_{21}(\delta) U = 0. \label{e: transcrit U eigenvalue problem} 
	\end{align}
	Since $g_1$ is quadratic, every term in $\partial_U g_1 (U_\mathrm{fr}, V_\mathrm{fr}; \delta^2)$ contains a factor of $U_\mathrm{fr}$ or $V_\mathrm{fr}$, and hence is exponentially decaying on the right at least as fast as $ye^{-y}$. It follows that
	\begin{align}
		\| \mathcal{A}_{21}(\delta) \|_{L^2_{\mathrm{exp}, 0, 1} \to L^2_{\mathrm{exp}, 0, \eta}} \leq C |\delta|^2. \label{e: transcrit A21 estimate}
	\end{align} 
	Also, $\partial_V g_0 (U_\mathrm{fr}, V_\mathrm{fr}; \delta^2)$ is bounded in space, and hence $\| \mathcal{A}_{12}(\delta) \|_{L^2_{\mathrm{exp}, 0, \eta} \to L^2_{\mathrm{exp}, 0, \eta}} \leq C$, and so we conclude
	\begin{multline*}
		\|  \mathcal{A}_{12}(\delta) (\mathcal{A}_{22}(\delta) - \gamma^2)^{-1} \mathcal{A}_{21}(\delta) U \|_{L^2_{\mathrm{exp}, 0, \eta}} \leq \| \mathcal{A}_{12}(\delta)\|_{L^2_{\mathrm{exp}, 0, \eta} \to L^2_{\mathrm{exp}, 0, \eta}} \| (\mathcal{A}_{22}(\delta) - \gamma^2)^{-1} \|_{L^2_{\mathrm{exp}, 0, \eta} \to L^2_{\mathrm{exp}, 0, \eta}} \\
		\hfill \cdot \| \mathcal{A}_{21}(\delta) \|_{L^2_{\mathrm{exp}, 0, 1} \to L^2_{\mathrm{exp}, 0, \eta}} \| U \|_{L^2_{0,1}} \\
		\leq C |\delta|^2 \| U \|_{L^2_{0, 1}}
	\end{multline*}
	for all $\delta$ sufficiently small. Hence we may view this term as a perturbation of the principal eigenvalue problem $(\mathcal{A}_{11}(\delta) - \gamma^2) U = 0$. 
	
	To solve this eigenvalue problem, we make the far-field/core ansatz
	\begin{align}
		U_\mathrm{eig}(y, \gamma) = w(y) + \alpha \chi_+ (y) e^{-(1+\gamma) y} =: w(y) + \alpha e_+(y, \gamma).
	\end{align}
	Inserting this ansatz into \eqref{e: transcrit U eigenvalue problem}, we obtain an equation
	\begin{align}
		\mathcal{F}(w, \alpha; \gamma, \delta) = 0,
	\end{align}
	where
	\begin{align}
		\mathcal{F}(w, \alpha; \gamma, \delta) = (\mathcal{A}_{11}(\delta) - \gamma^2) U_\mathrm{eig} - \mathcal{A}_{12}(\delta) (\mathcal{A}_{22}(\delta) - \gamma^2)^{-1} \mathcal{A}_{21}(\delta) U_\mathrm{eig}. 
	\end{align}
	\begin{lemma}\label{l: transcrit evans fcn regularity}
		Fix $\gamma_0$ and $\delta_0$ sufficiently small. The mapping $\mathcal{F}: H^2_{\mathrm{exp}, 0, \eta} \times \C \times B(0,\gamma_0) \times (-\delta_0, \delta_0) \to L^2_{\mathrm{exp}, 0, \eta}$ is well-defined, linear in $w$ and $\alpha$, analytic in $\gamma$, and continuous in $\delta$. Moreover, any $\gamma^2$ to the right of the essential spectrum of $\mcl(\delta) = \omega_{0,1} A(\delta) \omega_{0,1}^{-1}$ is an eigenvalue of $\mcl(\delta)$ if and only if there exist $(w, \alpha) \in H^2_{\mathrm{exp}, 0, \eta} \times \C$ such that $\mathcal{F}(w, \alpha; \gamma, \delta) = 0$ with $\Re \gamma > 0$. 
	\end{lemma}
	\begin{proof}
		That $\mathcal{F}$ preserves exponential localization follows from the fact that $\chi_+ e^{-(1+\gamma)y}$ solves $(\mathcal{A}_+(\delta) - \gamma^2) u = 0$ and that the nonlocal perturbation gains exponential localization by \eqref{e: transcrit A21 estimate}. Continuity in $\delta$ also follows from \eqref{e: transcrit A21 estimate}. Analyticity in $\gamma$ follows as in \cite[Proposition 5.11]{PoganScheel}, with the additional observation that $(\mathcal{A}_{22}(\delta) - \gamma^2)^{-1}$ is analytic in $\gamma^2$ in $L^2_{\mathrm{exp}, 0, \eta}$ by standard spectral theory. Equivalence to the standard eigenvalue problem follows as in \cite[proof of Proposition 5.11, step 6]{PoganScheel}. 
	\end{proof}
	
	We now perform a Lyapunov-Schmidt reduction, decomposing the eigenvalue problem into an invertible infinite dimensional part and a finite dimensional part which detects eigenvalues. We let $P: L^2_{\mathrm{exp}, 0, \eta} \to \Rg \mathcal{A}_{11}(0) \subset L^2_{\mathrm{exp}, 0, \eta}$ denote the $L^2$-orthogonal projection onto the range of $\mathcal{A}_{11}(0)$. It follows from Lemma \ref{l: transcrit hyp 1} that $\mathcal{A}_{11}(0) = \mathcal{A}_\mathrm{kpp} : H^2_{\mathrm{exp}, 0, \eta} \to L^2_{\mathrm{exp}, 0, \eta}$ is Fredholm with index $-1$. The proof of Lemma \ref{l: transcrit lin invertible} implies, in particular, that $\mathcal{A}_{11}(0)$ has trivial kernel and one-dimensional co-kernel, and we let $\ker \mathcal{A}_{11}(0)^* = \spn (\phi)$ for some function $\phi \in L^2_{\mathrm{exp}, 0, -\eta}$. We can then decompose the eigenvalue problem \eqref{e: transcrit U eigenvalue problem} as 
	\begin{align}
		\begin{cases}
			P \mathcal{F} (w, \alpha; \gamma, \delta) &= 0, \\
			\langle \mathcal{F}(w, \alpha;\gamma, \delta), \phi \rangle &= 0. 
		\end{cases} \label{e: transcrit eigenvalue decomposition}
	\end{align}
	This system has a trivial solution $(w, \alpha; \gamma, \delta) = (0, 0; 0, 0)$. The linearization of the first equation about this trivial solution is $P \mathcal{A}_{11}(0)$, which is invertible by construction, and so by the implicit function theorem we can solve the first equation for $w(\alpha; \gamma, \delta)$. Since the equation is linear in $\alpha$ and the implicit function theorem guarantees a unique solution in a neighborhood of the origin, we find that this solution $w$ must have the form
	\begin{align}
		w (\alpha; \gamma, \delta) = \alpha \tilde{w}(\gamma, \delta). 
	\end{align}
	Inserting this into the second equation of \eqref{e: transcrit eigenvalue decomposition} and eliminating the common factor of $\alpha$ in every term, we find a reduced scalar equation
	\begin{align}
		E(\gamma, \delta) := \langle (\mathcal{A}_{11}(\delta) - \gamma^2) (\tilde{w} + e_+) - \mathcal{A}_{12} (\delta) (\mathcal{A}_{22}(\delta) - \gamma^2)^{-1} \mathcal{A}_{21}(\delta) (\tilde{w} + e_+), \phi \rangle.
	\end{align}
	Lemma \ref{l: transcrit evans fcn regularity} implies that $E(\gamma, \delta)$ is analytic in $\gamma$ and continuous in $\delta$, and that $\mathcal{A}(\delta)$ has an eigenvalue $\gamma^2$ to the right of its essential spectrum if and only if $E(\gamma, \delta) = 0$ with $\Re \gamma > 0$. 
	
	\begin{prop}\label{p: transcrit point spectrum}
		For $\delta > 0$ sufficiently small, the operator $\mcl(\delta) = \omega_{0,1}A(\delta) \omega_{0,1}^{-1} : H^2 \times H^2 \to L^2 \times L^2$ has no eigenvalues with $\Re \lambda \geq 0$, and there is no bounded solution to $\mcl(\delta) u = 0$. That is, for $\delta > 0$ sufficiently small, \eqref{e: transcritical scaled 2} satisfies Hypothesis \ref{hyp: no unstable point spectrum}. 
	\end{prop}
	\begin{proof}
		Eigenvalues bifurcating out of the essential spectrum are tracked by zeros of $E(\gamma, \delta)$. We compute
		\begin{align}
			E(0, 0) = \langle \mathcal{A}_{11}(0) (\tilde{w}(0,0) + e_+(\cdot, 0)), \phi \rangle. 
		\end{align}
		It was shown in \cite[proof of Lemma 4.6]{AveryGarenaux} that $E(0,0) \neq 0$, and hence $E(\gamma, \delta)$ is nonzero for $\gamma, \delta$ small. This also implies that $\mcl(0) u = 0$ has no bounded solutions, as the existence of a bounded solution would imply $E(0,0) = 0$. 
		
		Away from the essential spectrum, which touches the imaginary axis only at the origin and is otherwise stable, the eigenvalue problem \eqref{e: transcrit U eigenvalue problem} is a regular perturbation of the Fisher-KPP eigenvalue problem $(\mathcal{A}_{11}(0) - \lambda) u = 0$, which has no eigenvalues with $\Re \lambda \geq 0$, and hence there are no eigenvalues for the full problem with $\Re \lambda \geq 0$ by standard spectral perturbation theory. See e.g. \cite[proof of Theorem 2]{AveryGarenaux} for further details. 
	\end{proof}
	
	Theorem \ref{t: transcritical} follows from Lemma \ref{l: transcrit hyp 1}, Corollary \ref{c: transcrit existence}. and Proposition \ref{p: transcrit point spectrum}. Applying Theorem \ref{t: selection}, we obtain the following description of invasion dynamics in the original system \eqref{e: transcritical 2}.
	
	\begin{corollary}\label{c: transcrit}
		Consider \eqref{e: transcritical 2} with $\mu > 0$ small. There exist open classes of steep initial data which evolve into front-like profiles propagating with the linear spreading speed $c_* = 2 \sqrt{\mu}$. 
	\end{corollary}
	
	\section{Saddle-node bifurcation}\label{s: saddle node}
	We now assume the reaction kinetics undergo a saddle-node bifurcation, with general form
	\begin{align}
		u_t &= u_{xx} + \mu - u^2 + f_0(u, v; \mu) \nonumber \\
		v_t &= D v_{xx} - K v + f_1 (u, v; \mu), \label{e: saddle node 2}
	\end{align}
	with $\mu > 0$, $(u, v) \in \R \times \R^{n-1}$, and positive matrices $D, K \in \R^{n-1} \times \R^{n-1}$ satisfying \eqref{e: no turing}. We assume $f_0$ and $f_1$ are smooth, with 
	\begin{align*}
		f_0(u, v;\mu) &= \mathrm{O}\left(  \mu u, \mu |v|, u |v|, |v|^2, \mu |v|, u^3 \right), \\
		f_1 (u, v; \mu) &= \mathrm{O}\left(\mu |v|, u^2, |v|^2, u |v|\right)
	\end{align*}
	as $u, |v|, \mu \to 0$. 
	Introducing the rescaled variables
	\begin{align}
		y = \mu^{1/4} x, \quad t = \mu^{1/2} \tau, \quad U(y,\tau) = \mu^{-1/2} u(x,t), \quad V(y,\tau) = \mu^{-1/2} v(x,t),
	\end{align}
	we find
	\begin{align*}
		U_\tau &= U_{yy} + 1 - U^2 + g_0(U, V; \mu) \\
		\mu^{1/2} V_\tau &= \mu^{1/2} D V_{yy} - K V + \mu^{1/2} g_1(U, V; \mu), 
	\end{align*}
	where
	\begin{align*}
		g_0 (U, V; \mu) &= \frac{1}{\mu} f_0 (\mu^{1/2} U, \mu^{1/2} V; \mu), \\
		g_1 (U, V; \mu) &= \frac{1}{\mu} f_1 (\mu^{1/2} U, \mu^{1/2} V; \mu)
	\end{align*}
	are smooth in $U, V$, and $\sqrt{\mu}$. The leading order equation is now $U_\tau = U_{yy} + 1 - U^2$, which has $U = 1$ and $U = -1$ as stable and unstable equilibria, respectively. Making the change of variables $U = W - 1$, we then find
	\begin{align}
		W_\tau &= W_{yy} + 2 W - W^2 + g_0 (W-1, V; \mu) \nonumber\\
		\mu^{1/2} V_\tau &= \mu^{1/2} D V_{yy} - K V + \mu^{1/2} g_1(W-1, V; \mu). \label{e: saddle-node scaled 2}
	\end{align}
	
	\begin{thmlocal}\label{t: saddle-node}
		For $\mu > 0$ sufficiently small, the system \eqref{e: saddle-node scaled 2} satisfies Hypotheses \ref{hyp: PDR} through \ref{hyp: no unstable point spectrum}. 
	\end{thmlocal}
	\begin{proof}
		The leading order equation $W_\tau = W_{yy} + 2W - W^2$ still admits pulled front solutions, with the linear spreading speed $c = 2 \sqrt{2}$. Regularizing the singular perturbation as in the proof of Theorem \ref{t: main}, we recover continuity in $\delta := \mu^{1/4}$, and so the result follows by the same argument. 
	\end{proof}
	
	\begin{corollary}
		Consider \eqref{e: saddle node 2} with $\mu > 0$ small. There exist open classes of steep initial data which evolve into front-like profiles propagating with the linear spreading speed $c_* = 2\sqrt{ 2} \mu^{1/4}$. 
	\end{corollary}
	
	\section{Supercritical pitchfork bifurcation}\label{s: pitchfork}
	We now consider a system in which the reaction kinetics undergo a supercritical pitchfork bifuraction, with general form
	\begin{align}
		u_t &= u_{xx} + \mu u - u^3 + f_0(u, v; \mu) \nonumber \\
		v_t &= D v_{xx} - K v + f_1 (u, v; \mu), \label{e: pitchfork unscaled 2}
	\end{align}
	for $(u,v) \in \R \times \R^{n-1}$ with small parameter $\mu > 0$. We again assume that $D, K \in \R^{n-1 \times n-1}$ and have strictly positive eigenvalues and satisfy \eqref{e: no turing}. We assume that $f_0$ and $f_1$ are smooth, with
	\begin{align*}
		f_0 (u, v; \mu) &= \mathrm{O}\left( u^4, u|v|^2, u^2 |v|, |v|^3, \mu u^2, \mu |v|^2 \right), \\
		f_1(u, v; \mu) &= \mathrm{O}\left(\mu |v|, u^2, |v|^2, u|v|\right)
	\end{align*}
	as $u, |v|, \mu \to 0$.
	Introducing the rescaled variables
	\begin{align}
		y = \sqrt{\mu} x, \quad \tau = \mu t, \quad U(y,\tau) = \mu^{-1/2} u(x,t), \quad V(y, \tau) = \mu^{-1/2} v(x,t),
	\end{align}
	we find
	\begin{align}
		U_\tau &= U_{yy} + U - U^3 + g_0(U, V; \mu), \nonumber\\
		\mu V_\tau &= \mu D V_{yy} - K V + \mu^{1/2} g_1 (U, V; \mu) \label{e: pitchfork 2},
	\end{align}
	where
	\begin{align*}
		g_0 (U, V; \mu) &= \mu^{-3/2} f_0 (\mu^{1/2} U, \mu^{1/2} V; \mu) \\
		g_1 (U, V; \mu) &= \mu^{-1} f_1 (\mu^{1/2} U, \mu^{1/2} V; \mu)
	\end{align*}
	are smooth in $U, V$, and $\sqrt{\mu}$.  
	\begin{thmlocal}\label{t: pitchfork}
		For $\mu > 0$ sufficiently small, the system \eqref{e: pitchfork 2} satisfies Hypotheses \ref{hyp: PDR} through \ref{hyp: no unstable point spectrum}.
	\end{thmlocal}
	\begin{proof}
		Since the leading order equation $U_\tau = U_{yy} + U - U^3$ still admits pulled fronts connecting $U = 1$ to $U= 0$, and the term $\mu^{1/2} g_1 (U, V; \mu)$ is still continuous in $\delta := \sqrt{\mu}$, the proof is exactly the same as that of Theorem \ref{t: transcritical}.  
	\end{proof}
	\begin{corollary}
		Consider \eqref{e: pitchfork unscaled 2} with $\mu > 0$ small. There exist open classes of steep initial data which evolve into front-like profiles propagating with the linear spreading speed $c_* = 2 \sqrt{\mu}$. 
	\end{corollary}

	\bibliography{bifurcations}
	
\end{document}